\documentclass[12pt,reqno]{amsart}
\usepackage{hyperref}
\usepackage{amsmath,amssymb}
\usepackage[english]{babel}
\usepackage{caption}
\usepackage{amssymb,amsmath,euscript,enumerate,tikz}
\usepackage[margin=1in]{geometry}
\usepackage{graphicx}
\usepackage{float}
\usepackage{pgf,tikz}
\usepackage{mathrsfs}
\usepackage{upgreek}

\newcommand \reg{\operatorname{reg}}

\newcommand \Tor{\operatorname{Tor}}

\newcommand \pd{\operatorname{projdim}}

\newcommand \ca{\operatorname{c}}
\newcommand \ba{\operatorname{b}}
\newcommand \Ba{\operatorname{B}}
\newcommand\ck{\operatorname{c_k}}
\newcommand\cj{\operatorname{c_j}}
\newcommand\cab{\operatorname{c_4}}

\newcommand \K{\mathbb{K}}
\newcommand{\T}{\operatorname{Tor}}

\newtheorem{theorem}{Theorem}[section]

\newtheorem{lemma}[theorem]{Lemma}

\newtheorem{corollary}[theorem]{Corollary}

\begin{document}
\title[Bound for the regularity of Cactus Graphs]{Bound for the regularity of binomial edge ideals of cactus graphs}
\author[A. V. Jayanthan]{A. V. Jayanthan}
\email{jayanav@iitm.ac.in}
\author[Rajib Sarkar]{Rajib Sarkar$^1$}
\email{rajib.sarkar63@gmail.com}
\thanks{$^1$ The author is funded by University Grants Commission,
	India}
\address{Department of Mathematics, Indian Institute of Technology
	Madras, Chennai, INDIA - 600036}

\begin{abstract}
In this article, we obtain an upper bound for the regularity of the
binomial edge ideal of a graph whose every block is either a cycle or
a clique. As a consequence, we obtain an upper bound for the 
regularity of binomial edge ideal of a cactus graph. We also identify
certain subclass attaining the upper bound.
\end{abstract}
\keywords{Binomial edge ideal, Cactus graph, Cohen-Macaulayness, Regularity}
\thanks{AMS Subject Classification (2010): 13D02,13C13, 05E40}
\maketitle
\section{Introduction}
Let $R = \K[x_1,\ldots,x_m]$ be the standard graded polynomial ring over an arbitrary
field $\K$ and $M$ be a finitely generated graded  $R$-module. 
Let
\[
0 \longrightarrow \bigoplus_{j \in \mathbb{Z}} R(-p-j)^{\beta_{p,p+j}(M)} 
{\longrightarrow} \cdots {\longrightarrow} \bigoplus_{j \in \mathbb{Z}} R(-j)^{\beta_{0,j}(M)} 
{\longrightarrow} M\longrightarrow 0,
\]
be the minimal graded free resolution of $M$. The number $\beta_{i,j}(M)$ is called the 
$(i,j)$-th \textit{graded Betti number} of $M$.  The 
\textit{projective dimension} of $M$, denoted by $\pd(M)$, is defined as 
\[\pd(M):=\max\{i : \beta_{i,i+j}(M) \neq 0\}\]
and the \textit{Castelnuovo-Mumford regularity} (or simply, \textit{regularity}) of $M$, denoted by $\reg(M)$, is defined as 
\[\reg(M):=\max \{j : \beta_{i,i+j}(M) \neq 0\}. \]
The Betti number $\beta_{i,i+j}(M)$ is called an \textit{extremal Betti number} if $\beta_{i,i+j}(M)\neq 0$ and for all pairs of integers $(k,l)\neq (i,j)$, $\beta_{k,k+l}(M)=0$ where $k\geq i$ and $l\geq j$. If $p =\pd(M)$ and $r=\reg(M)$, then $M$ admits unique extremal Betti number if and only if 
$\beta_{p,p+r}(M)\neq 0$.
 
For a graph $G$, on $[n]$, the \textit{binomial edge ideal} of $G$ is the ideal
$J_G = \langle x_iy_j - x_j y_i : \{i, j\} \in E(G) \rangle \subset
S = \K[x_1, \ldots, x_n, y_1, \ldots, y_n]$, \cite{HH1, oh}. There
has been a lot of research in understanding the connection between
algebraic properties of $J_G$ and combinatorial properties of $G$. In
particular, researchers have been trying to understand the connection
between algebraic invariants of $J_G$, such as Betti numbers,
Castelnuovo-Mumford regularity, projective dimension, in terms of the
combinatorial invariants of $G$ such as number of maximal cliques in
$G$, length of maximal induced path. In this
paper, we deal with the regularity of $J_G$.

It has been conjectured by Saeedi Madani and Kiani that for a graph
$G$, $\reg(S/J_G) \leq \ca(G)$, where $\ca(G)$ denotes the number of
maximal cliques in $G$, \cite{KM2}. This conjecture has been proved
for only a few classes of graphs, see
\cite{EZ,JA1,Kahle19,KM-arxiv13,Arv-Jaco,MM,KM18,KM1}.
A connected graph $G$ is said to be a
\textit{cactus} graph if every block of $G$ is either a cycle or an
edge (see Section 2 for details). In this
article, we obtain an upper bound, which is finer than
Saeedi Madani-Kiani conjectured upper bound, for the regularity of
cactus graphs. For a cactus graph having a lot of cycles as blocks, it
turns out that the invariant $\ca(G)$ is much larger than the
upper bound that we have obtained. We also prove that the upper bound
is attained by a subclass of Cohen-Macaulay cactus graphs. In fact, we
prove our results for a slightly larger class of graphs. In the next
section we recall the necessary definitions and some of the results
from the literature which are crucial to the proofs of main results. In
Section 3, we prove the upper bound and in Section 4, we identify a
class of graphs for which the regularity upper bound is attained.
\section{Preliminaries}
Let us recall some basic definitions and notation from graph theory, which will be used throughout the article.
Let $G$ be a simple graph on the vertex set $V(G)=[n]:=\{1,\dots,n\}$ and the edge set $E(G)$.
A graph $G$ on $[n]$ is called a \textit{complete graph}, if
$\{i,j\} \in E(G)$ for all $1 \leq i < j \leq n$. We denote the complete graph on $n$ vertices by $K_n$. For $A \subseteq V(G)$, the
\textit{induced subgraph} of $G$ on the vertex set $A$, denoted by $G[A]$, is the graph such that for any pair of vertices
$i,j \in A$, $\{i,j\} \in E(G[A])$ if and only if $ \{i,j\} \in
E(G)$.  For a vertex $v\in V(G)$, $G \setminus v$ denotes the  induced
subgraph of $G$ on the vertex set $V(G) \setminus \{v\}$.
For
$U\subseteq V(G)$, $U$ is called a \textit{clique} if the induced subgraph $G[U]$ is the complete graph. A vertex $v\in V(G)$ is said to be a \textit{simplicial vertex} if there is only one maximal clique containing $v$. If $v$ is not a simplicial vertex, then $v$ is called an \textit{internal vertex}. For a vertex $v\in V(G)$, let
$N_G(v) := \{u \in V(G) :  \{u,v\} \in E(G)\}$ denote the
\textit{neighborhood} of $v$ in $G$ and  $G_v$ is the graph with the vertex set
$V(G)$ and edge set $E(G_v) =E(G) \cup \{ \{u,w\}: u,w \in N_G(v)\}$. A graph with the vertex set $[n]$ and the edge set $\{i,i+1:1\leq i\leq n-1\}$ is called a \textit{path} and it is denoted by $P_n$. A
\textit{cycle} on the vertex set $[n]$, denoted by $C_n$, is a graph
with the edge set $\{i,i+1:1\leq i\leq n-1\}\cup \{1,n\}$ for $n\geq
3$. A vertex $v\in V(G)$ is called a \textit{cut vertex} if
$G\setminus v$ has more number of components than that of $G$. A
maximal connected subgraph of $G$, which has no cut vertex is called a
\textit{block} of $G$. A graph $G$ is called a \textit{block graph} if every block of $G$ is a clique. For a block graph $G$, a block of $G$ is called a leaf of $G$ if that block contains at most one cut vertex. A connected graph $G$ is said to be a
\textit{cactus} graph if every block of $G$ is either a cycle or an
edge. The \textit{block graph} of $G$, denoted by $\Ba(G)$, is the graph whose vertices are the blocks of $G$ and two vertices are adjacent whenever the corresponding blocks have a common cut vertex. The graph 
on the vertex set $[4]$ and with $E(G)=E(C_4)\cup
\{\{1,3\} \}$ or $E(G)=E(C_4)\cup \{\{2,4\} \}$ is called a
\textit{diamond} graph and is denoted by $D$.	If $G_1(\neq K_m)$ and
$G_2$ are two subgraphs of a graph $G$ such that $G_1\cap
G_2=K_m$, $V(G_1)\cup V(G_2)=V(G)$ and $E(G_1)\cup E(G_2)=E(G)$, then
$G$ is called the \textit{clique sum} of $G_1$ and $G_2$ along the
complete graph $K_m$, denoted by $G_1\cup_{K_m} G_2$. If the clique
sum is along an edge $e$, then we write $G_1 \cup_e G_2$ and if the
clique sum is along a vertex $v$, then we write $G_1 \cup_v G_2$.

Throughout this article, $S$ denotes the polynomial ring over $\K$ in
$2|V(G)|$ number of variables and in which the binomial edge ideal
$J_G$ resides. For a graph $H$, we denote by $S_H$, the polynomial
ring over $\K$ in $2|V(H)|$ number of variables and $J_H \subset S_H$.
If $H$ is a subgraph of $G$, we would assume that $S_H$ is a subring
of $S$. We say that $G$ is a \textit{Cohen-Macaulay graph} if $S/J_G$ is
Cohen-Macaulay.

Let $A=\K[x_1,\dots,x_m]$, $A'=\K[y_1,\dots,y_n]$ and $B=\K[x_1,\dots,x_m,y_1,\dots,y_n]$ be
polynomial rings. Let $I\subseteq A$ and $J\subseteq A'$ be
homogeneous ideals. Then the minimal free resolution of $B/(I+J)$ can
be obtained by the tensor product of the minimal free resolutions of
$A/I$ and $A'/J$. Therefore, for all $i,j$, we get:

\begin{align}\label{Bettiproduct}
\beta_{i,i+j}\left(\frac{B}{I+J}\right) =
\underset{{\substack{i_1+i_2=i \\
j_1+j_2=j}}}{\sum}\beta_{i_1,i_1+j_1}\left(\frac{A}{I}\right)\beta_{i_2,i_2+j_2}\left(\frac{A'}{J}\right).
\end{align}
The following lemma can be easily derived from the long exact sequence
of $\T$ corresponding to given short exact sequence.
\begin{lemma}\label{regularity-lemma}
	Let $R$ be a standard graded ring  and $M,N,P$ be finitely generated graded $R$-modules. 
	If $ 0 \rightarrow M \xrightarrow{f}  N \xrightarrow{g} P \rightarrow 0$ is a 
	short exact sequence with $f,g$  
	graded homomorphisms of degree zero, then 
	\begin{enumerate}[(i)]
		\item $\reg(M)\leq \max \{\reg(N),\reg(P)+1 \}$,
		\item $\reg (M) = \reg (N)$ if  $\reg (N) > \reg (P)$.
		
	\end{enumerate}	
\end{lemma}	
We recall the following lemma due to Ohtani.
\begin{lemma}$($\cite[Lemma 4.8]{oh}$)$\label{ohtani-lemma}
Let $G$ be a  graph on $V(G)$ and $v\in V(G)$ such that $v$ is an internal vertex. Then $J_G$ can be written as
	$$J_{G}=J_{G_v}\cap Q_v, \text{ where } Q_v = (x_v, y_v) + J_{G\setminus v}.$$
\end{lemma}
Rinaldo proved that 
$Q_v=\underset{{\substack{T\in \mathscr{C}(G), v
\in T}}} \cap P_{T}(G)$, \cite[Corollary 1.1]{Rinaldo-Cactus}.
\section{Regularity of Cactus Graph}
In this section we consider graphs whose every block is either a cycle
or a clique. For a graph $G$, we denote by $\ba(G)$ the number of blocks
of $G$ and by $\mathcal{C}(G)$ the number of cycles of length $\geq
4$. Let $\cj(G)$ denote the number of
cycles of length $j$ in $G$. Then we have $\mathcal{C}(G)=\sum_{k\geq
4}\ck(G)$. First, we prove a technical property that is required in
the main proof.
\begin{lemma}\label{tech-lemma}
Let $G$ be a graph with $\mathcal{C}(G) \geq 1$ and such that each block of $G$ is either a cycle or a clique. Then there exists a cycle $C$ of length $r(\geq 4)$ such that $r-2$ consecutive vertices of $C$ are not part of any other cycles of length $\geq 4$ in $G$.
\end{lemma}
\begin{proof}
We prove our assertion by induction on the number of blocks of $G$. If
$\ba(G)=1$, then we are through. Suppose now that $\ba(G)\geq 2$. It
follows from \cite[Theorem 3.5]{haray} that $\Ba(G)$ is block graph.
Set $\ba(\Ba(G))=s$. Let $B_1,\dots,B_s$ be the blocks of $\Ba(G)$ and
assume that $B_s$ is a leaf of $\Ba(G)$. There exist blocks $H_1,\dots,H_l$, for some $l\geq 2$, such that
$H=H_1\cup_{v}\dots\cup_{v}H_l$, for some $v \in V(G)$, is an induced
subgraph of $G$ corresponding to the block $B_s$.
Since $B_s$ is a leaf of $\Ba(G)$, it has only one cut-vertex. This
implies that there exists an $i\in [l]$
such that for any $j \neq i$, the vertices of $V(H_j),$
except $v$, are not part of any other blocks of $G$. 
If for some $j\neq i$, $H_j=K_m$,  then set $G'$ to be the induced
subgraph on $(V(G) \setminus V(H_j) ) \cup \{v\}$, i.e., the graph
obtained by removing the block $H_j$ from $G$.
Since
$\ba(G')=\ba(G)-1$, by induction there exists a cycle $C$ of
length $r(\geq 4)$ such that $r-2$ consecutive vertices of $C$
are not part of the other cycles of length $\geq 4$ in the graph
$G'$. Since $G'$ is an induced subgraph of $G$ and $V(C) \cap
V(H_j) \subseteq \{v\}$, we see that $C$ is a cycle in $G$ with
the required property. Suppose $H_j$ is not a complete graph for any
$j \neq i$, then taking $C = H_j$, we see that the
cycle $C$ satisfies the required property.
\end{proof}
We now prove the main theorem of this article, an upper bound for the
regularity of binomial edge ideal of a graph whose each block is
either a cycle or a clique.
\begin{theorem}\label{generalized-cactus-graph}
Let $G$ be a graph such that each block of $G$ is either a cycle or a clique. Then 
\[ \reg(S/J_G)\leq \ca'(G)+\sum_{k\geq 4}(k-2)\ck(G), \]
where $\ca'(G)$ is the number of maximal cliques except the edges of any
cycle of length $\geq 4$ in $G$
\end{theorem}
\begin{proof}
We prove our assertion by induction on $\mathcal{C}(G)$. If
$\mathcal{C}(G)=0$, then $G$ is a block graph, and hence the
assertion follows from \cite[Theorem 3.9]{EZ}. Assume that
$\mathcal{C}(G)\geq 1$. Then by Lemma \ref{tech-lemma}, there
exists a cycle $C$ of length $r\geq 4$ such that $r-2$ consecutive
vertices of $C$ are not part of any other cycles of length $\geq
4$ in $G$. Suppose $V(C)=\{v_1,\dots,v_r\}$ and
$v_1,\dots,v_{r-2}$ are not part of any other cycles of length
$\geq 4$ . Now, we bring in an intermediate class of graphs which is a
clique sum of $G$ with a complete graph $K_m$, for some $m \geq 2$.
Let $H=G\cup_{\{v_1,v_2\}} K_m$ for $m\geq 2$. Note that $H = G$ if $m
= 2$.
\vskip 2mm \noindent
\textbf{Claim:} $\reg(S_H/J_H)\leq \ca'(H)+\sum_{k\geq 4}(k-2)\ck(H).$
\vskip 2mm
We prove this claim for all $m \geq 2$ and finally deduce the required
assertion on $G$ by taking $m = 2$.
We proceed by induction on $r$. Suppose now that $r=4$. Since
$v_1$ is an internal vertex of $H$, by Lemma \ref{ohtani-lemma},
we can write $J_{H}=J_{H_{v_1}}\cap Q_{v_1}$ with
$Q_{v_1}=(x_{v_1},y_{v_1})+J_{H\setminus v_1}$, where $H\setminus v_1$
is a graph whose every block is either a cycle or a clique with $\mathcal{C}(H\setminus
v_1)=\mathcal{C}(G)-1$. Therefore, by induction and
\eqref{Bettiproduct}, 
\begin{align*}
\reg(S_H/Q_{v_1})\leq \ca'(H\setminus v_1)+\sum_{k\geq
	4}(k-2)\ck(H\setminus v_1)\leq \ca'(H)+\sum_{k\geq 4}(k-2)\ck(H),
\end{align*}
where the last inequality follows because $\ca'(H \setminus v_1)
\leq \ca'(H) + 2$. Note that $J_{H_{v_1}}+Q_{v_1}=(x_{v_1},y_{v_1})+J_{H_{v_1}\setminus v_1}$.
Let $n_{v_1}=|N_H(v_1)|, B'=C_3\cup _{\{v_2,v_4\}} K_{n_{v_1}+1}$ and $B''=C_3\cup _{\{v_2,v_4\}} K_{n_{v_1}}$. Then it can be seen
that the block containing $v_1$ in $H_{v_1}$ is $B'$ and all the
other blocks of $H_{v_1}$ are blocks of $H$. 
We show that 
\begin{eqnarray}
\reg(S_H/J_{H_{v_1}})& \leq & \ca'(H_{v_1})+\sum_{k\geq
4}(k-2)\ck(H_{v_1}) \text{ and } \label{eqn2}\\
\reg(S_H/((x_{v_1},y_{v_1})+J_{H_{v_1}\setminus v_1})) & \leq & 
\ca'(H_{v_1}\setminus v_1)+\sum_{k\geq 4}(k-2)\ck(H_{v_1}\setminus
v_1). \label{eqn3}
\end{eqnarray}

Let $H'=H_{v_1}$.
Since $v_2$ is an internal vertex in $H'$, by Lemma
\ref{ohtani-lemma}, $J_{H'}=J_{H'_{v_2}}\cap Q_{v_2}$ with
$Q_{v_2}=(x_{v_2},y_{v_2})+J_{H'\setminus v_2}$, where $H'\setminus
v_2$ is a graph such that each block of $H'\setminus v_2$ is either a cycle or a clique and $\mathcal{C}(H'\setminus
v_2)=\mathcal{C}(G)-1$. Therefore, by induction and
\eqref{Bettiproduct}, we have $$\reg(S_H/Q_{v_2})\leq
\ca'(H'\setminus v_2)+\sum_{k\geq 4}(k-2)\ck(H'\setminus v_2)\leq
\ca'(H')+\sum_{k\geq 4}(k-2)\ck(H'),$$ 
where the second inequality follows since $\ca'(H'\setminus v_2)\leq \ca'(H')$ and $\sum_{k\geq 4}(k-2)\ck(H'\setminus v_2)=\sum_{k\geq 4}(k-2)\ck(H').$
The block containing $v_2$ in $H'_{v_2}$ is $K_{n_{v_2}+1}$, where
$n_{v_2}=|N_{H'}(v_2)|$ and all other blocks of $H'_{v_2}$ are blocks
of $H'$. Therefore, $\mathcal{C}(H'_{v_2})=\mathcal{C}(G)-1$. By
induction, 
$$\reg(S_H/J_{H'_{v_2}})\leq \ca'(H'_{v_2})+\sum_{k\geq
4}(k-2)\ck(H'_{v_2})\leq \ca'(H')-1 + \sum_{k\geq 4}(k-2)\ck(H'),$$ 
where the second inequality follows since $\ca'(H'_{v_2})\leq \ca'(H')-1$
and $\sum_{k\geq 4}(k-2)\ck(H'_{v_2})=\sum_{k\geq
4}(k-2)\ck(H').$
Note that $J_{H'_{v_2}}+Q_{v_2}=(x_{v_2},y_{v_2})+J_{H'_{v_2}\setminus
v_2}$. Therefore, by \cite[Corollary 2.2]{MM} and
\eqref{Bettiproduct}, 
$$\reg(S_H/((x_{v_2},y_{v_2})+J_{H'_{v_2}\setminus v_2}))\leq
\reg(S_H/J_{H'_{v_2}}) \leq \ca'(H')-1+\sum_{k\geq
4}(k-2)\ck(H').$$
We consider the following short exact sequence:
\begin{align*}\label{ses-v2}
0\longrightarrow \frac{S_H}{J_{H'}}\longrightarrow
\frac{S_H}{J_{H'_{v_2}}}\oplus
\frac{S_H}{Q_{v_2}}\longrightarrow
\frac{S_H}{J_{H'_{v_2}}+Q_{v_2}} \longrightarrow 0.
\end{align*}
By applying Lemma \ref{regularity-lemma} on the above short exact
sequence, we get that
\[
\reg(S_H/J_{H'})\leq \ca'(H')+\sum_{k\geq 4}(k-2)\ck(H').
\]
Since $\ca'(H')\leq \ca'(H)+2$
and $\sum_{k\geq 4}(k-2)\ck(H')=\sum_{k\geq 4}(k-2)\ck(H)-2$, we have
$$\reg(S_H/J_{H'})\leq
\ca'(H)+\sum_{k\geq 4}(k-2)\ck(H).$$ 

\ \ \ \ \ \ \

Now we prove \eqref{eqn3}.
If $n_{v_1}=2$, then $H_{v_1}\setminus v_1$ is a graph whose blocks
are either a cycle or a clique with $\mathcal{C}(H_{v_1}\setminus
v_1)=\mathcal{C}(G)-1$. Therefore by induction and
\eqref{Bettiproduct},
$$\reg(S_H/((x_{v_1},y_{v_1})+J_{H_{v_1}\setminus v_1}))\leq
\ca'(H_{v_1}\setminus v_1)+\sum_{k\geq 4}(k-2)\ck(H_{v_1}\setminus
v_1).$$
Now assume that $n_{v_1}\geq 3$. Set $H''=H_{v_1}\setminus v_1$.
Replacing $H'$ by $H''$ in the proof of (\ref{eqn2}), 
and obtain
$$\reg(S_H/((x_{v_1},y_{v_1})+J_{H_{v_1}\setminus v_1}))\leq
\ca'(H_{v_1}\setminus v_1)+\sum_{k\geq 4}(k-2)\ck(H_{v_1}\setminus
v_1).$$ 
For
$H'\setminus v_1$, it can be observed that $\ca'(H'\setminus v_1)\leq
\ca'(H)+1$ and $\sum_{k\geq 4}(k-2)\ck(H')=\sum_{k\geq 4}(k-2)\ck(H)-2$. Therefore,
 $$\reg(S_H/((x_{v_1}, y_{v_1}) + J_{H'\setminus v_1})) \leq \ca'(H)-1+\sum_{k\geq 4}(k-2)\ck(H).$$
Now the claim follows from Lemma \ref{regularity-lemma} applied on the
short exact sequence:
\begin{align*}
0\longrightarrow \frac{S_H}{J_{H}}\longrightarrow
\frac{S_H}{J_{H'}}\oplus
\frac{S_H}{Q_{v_{1}}}\longrightarrow
\frac{S_H}{J_{H'}+Q_{v_{1}}} \longrightarrow 0.
\end{align*}
Assume that $r\geq 5$. Since
$v_1$ is an internal vertex of $H$, by Lemma \ref{ohtani-lemma},
we write $J_{H}=J_{H_{v_1}}\cap Q_{v_1}$ with
$Q_{v_1}=(x_{v_1},y_{v_1})+J_{H\setminus v_1}$, where $J_{H_{v_1}}+Q_{v_1}=(x_{v_1},y_{v_1})+J_{H_{v_1}\setminus v_1}$. Since $H\setminus v_1$
is a graph such that every block of $H\setminus v_1$ is either a cycle or a clique and $\mathcal{C}(H\setminus
v_1)=\mathcal{C}(G)-1$, by induction and
\eqref{Bettiproduct}, 
\begin{align*}
\reg(S_H/Q_{v_1})\leq \ca'(H\setminus v_1)+\sum_{k\geq
4}(k-2)\ck(H\setminus v_1)\leq \ca'(H)+\sum_{k\geq 4}(k-2)\ck(H),
\end{align*}
where the last inequality follows because $\sum_{k\geq
4}(k-2)\ck(H\setminus v_1)+r-2=\sum_{k\geq 4}(k-2)\ck(H)$ and $\ca'(H
\setminus v_1) \leq \ca'(H) + r-2$.
Let $n_{v_1}=|N_H(v_1)|, B'=C_{r-1}\cup _{\{v_2,v_r\}} K_{n_{v_1}+1}$
and $B''=C_{r-1}\cup _{\{v_2,v_r\}} K_{n_{v_1}}$. Then, the block
containing $v_1$ in $H_{v_1}$ is $B'$ and all the
other blocks of $H_{v_1}$ are blocks of $H$. Note that both the graphs
$H_{v_1}$ and $H_{v_1}\setminus v_1$ have a cycle of length $r-1$ whose
$r-3$ vertices are not part of any other cycles. Hence by induction on
$r$ and using \eqref{Bettiproduct}, we get
$$\reg(S_H/J_{H_{v_1}})\leq \ca'(H_{v_1})+\sum_{k\geq
4}(k-2)\ck(H_{v_1}) \text{ and }$$
$$\reg(S_H/((x_{v_1},y_{v_1})+J_{H_{v_1}\setminus v_1}))\leq
\ca'(H_{v_1}\setminus v_1)+\sum_{k\geq 4}(k-2)\ck(H_{v_1}
\setminus v_1).$$ 
Note also that $\sum_{k\geq
4}(k-2)\ck(H_{v_1})=\left[\sum_{k\geq
4}(k-2)\ck(H) \right]-1$ and $\ca'(H_{v_1})\leq \ca'(H)+1$. Therefore, $$\reg(S_H/J_{H_{v_1}})\leq \ca'(H)+\sum_{k\geq
4}(k-2)\ck(H).$$
For the graph $H_{v_1}\setminus v_1$, we notice that $\sum_{k\geq
4}(k-2)\ck(H_{v_1}\setminus v_1)=\left[\sum_{k\geq
4}(k-2)\ck(H)\right]-1$ and $\ca'(H_{v_1}\setminus v_1)\leq \ca'(H)$. Therefore, 
$$\reg(S_H/((x_{v_1},y_{v_1})+J_{H_{v_1}\setminus v_1}))\leq \ca'(H)+\left[\sum_{k\geq
4}(k-2)\ck(H)\right]-1.$$
Hence it follows from 
Lemma \ref{regularity-lemma} applied on the following short exact sequence 
\begin{align*}
0\longrightarrow \frac{S_H}{J_{H}}\longrightarrow
\frac{S_H}{J_{H_{v_1}}}\oplus
\frac{S_H}{Q_{v_1}}\longrightarrow
\frac{S_H}{J_{H_{v_1}}+Q_{v_1}} \longrightarrow 0.
\end{align*}
that $\reg(S_H/J_H) \leq \ca'(H) + \sum_{k\geq 4} (k-2)\ck(H)$. This
completes the proof of the Claim.
Hence, our assertion follows by taking $m=2$ in $H$.
\end{proof}
As an immediate consequence, we obtain an upper bound for the
regularity of cactus graph.
\begin{corollary}\label{cactus-graph}
Let $G$ be a cactus graph. Then 
\[\reg(S/J_G) \leq \ca'(G) + \sum_{k\geq 4} (k-2)\ck(G).\] 
\end{corollary}
\section{Regularity of Cohen-Macaulay Cactus Graph}
In this section we obtain a class of cactus graph for which the upper
bound we proved in the last section is attained. We begin by computing
the regularity of certain classes of graphs which are required in the
main theorem.
\begin{lemma}\label{tech-lemma1}
Let $k,m_1\geq 3$ and $m_2\geq 2$. Let $G = C_k \cup_e K_{m_1} \cup_v
K_{m_2}$ with $v \in e$. If $m_2=2$, then $\reg(S/J_G)=k-1$ and if
$m_2\geq 3 $, then $\reg(S/J_G)=k$.
\end{lemma}
\begin{proof}
Let $N_{C_k}(v)=\{u,w\}$ and $e=\{u,v\}$. If $m_2=2$, then the
assertion follows from the proof \cite[Theorem 4.1]{Rajib}.
Suppose now that $m_2\geq 3$. Then using Lemma
\ref{ohtani-lemma}, we get that $J_{G}=J_{G_v}\cap Q_v$ with
$Q_v=(x_v,y_v)+J_{G\setminus v}$, where $G\setminus v$ is the
graph with two components $K_{m_1-1}\cup_u P_{k-1}$
and $K_{m_2-1}$. By \cite[Theorem 3.1]{JNR1} and
\eqref{Bettiproduct}, $\reg(S/Q_v)=k$. Note that
$G_v=C_{k-1}\cup_{e'}K_{m_1+m_2}$, where $e'=\{u,w\}$. Also,
$J_{G_v}+Q_v=(x_v,y_v)+J_H$, where $H$ is obtained by deleting the
vertex $v$ from $G_v$, i.e., $H=C_{k-1}\cup_{e'}K_{m_1+m_2-1}$.
By \cite[Theorem 3.12]{JAR2} and \eqref{Bettiproduct}, we have
$\reg(S/J_{G_v})=\reg(S/(J_{G_v}+Q_v))=k-2$. As $v$ is not a
simplicial vertex, we consider the following short exact sequence:
\begin{align}\label{rinaldo-ses}
0\longrightarrow \frac{S}{J_{G}}\longrightarrow  \frac{S}{J_{{G_v}}}\oplus \frac{S}{Q_v} \longrightarrow \frac{S}{J_{G_v}+Q_v} \longrightarrow 0.
\end{align}
Therefore, it follows from Lemma \ref{regularity-lemma} and the above short exact sequence that $\reg(S/J_G)=k$.
\end{proof}
\begin{lemma}\label{tech-lemma2}
Let $k\geq 4$ and $m_1,m_2\geq 2$. Let $G = C_k \cup_u K_{m_1} \cup_v
K_{m_2}$ for some $\{u, v\} \in E(C_k)$. If $m_1=m_2=2$, then
$\reg(S/J_G)=k-1$, otherwise $\reg(S/J_G)=k$.
\end{lemma}
\begin{proof}
If $m_1=m_2=2$, then the assertion follows from \cite[Proposition 3.13]{JAR2}. Suppose $m_1\geq 3$ or $m_2\geq 3$. We assume that $m_2\geq 3$. Let $G=C_k\cup_u K_{m_1}\cup_v K_{m_2}$, $m_1\geq 2$, $m_2\geq 3$ and $N_{C_k}(v)=\{u,w\}$. Note that $G_v=C_{k-1}\cup _e K_{m_2+2}\cup_u K_{m_1}$, where $e=\{u,w\}$ and $Q_v=(x_v,y_v)+J_{K_{m_2-1}}+J_{H}$, where $H=K_{m_1}\cup_u P_{k-1}$. Also, $J_{G_v}+Q_v=(x_v,y_v)+J_{H'}$, where $H'$ is obtained by deleting the vertex $v$ from $G_v$, i.e., $H'=C_{k-1}\cup_e K_{m_2+1}\cup_u K_{m_1}$. By \cite[Theorem 3.1]{JNR1} and \eqref{Bettiproduct}, $\reg(S/Q_v)=k$. By Lemma \ref{tech-lemma1}, we get that if $m_1=2$, then $\reg(S/J_{G_v})=\reg(S/(J_{G_v}+Q_v))=k-2$, otherwise $\reg(S/J_{G_v})=\reg(S/(J_{G_v}+Q_v))=k-1$. Hence, the assertion follows from Lemma \ref{regularity-lemma} and the short exact sequence \eqref{rinaldo-ses}.
\end{proof}
A graph $G$ is said to be a \textit{decomposable graph} if $G$ can be written
as a clique sum of two subgraphs along a simplicial vertex i.e.,
$G=G_1\cup_v G_2$, where $v$ is a simplicial vertex of $G_1$ and
$G_2$. If $G$ is not decomposable, then it is called an \textit{indecomposable
graph}. It follows from \cite[Theorem 3.1]{JNR1} that
to find the regularity, it is
enough to consider $G$ to be an indecomposable graph. So, for the rest
of the section, we assume that $G$ is an indecomposable graph.
We now study the regularity of binomial edge ideal of Cohen-Macaulay
cactus graphs. Let $G$ be a graph such that $\Ba(G)$ is a path of length
$l-1$. Let $V(\Ba(G))=\{B_1,\dots,B_l\}$. If $B_i$ is a graph on $m_i$
vertices, then set $V(B_i)=\{v_{i1},\dots,v_{im_i}\}$ and $B_i\cap
B_{i+1}=\{w_i\}$. Also, we choose the order of vertices in $V(B_i)$ in
such a way that $v_{im_i}=w_i=v_{i+11}$.

In \cite[Theorem 2.2]{Rinaldo-Cactus}, Rinaldo characterized Cohen-Macaulay cactus graph. Let $G$ be an indecomposable Cohen-Macaulay cactus graph whose blocks are $B_1,\cdots,B_l$. Then it follows from \cite[Lemma 2.3]{Rinaldo-Cactus} that either $G\in \{K_2,C_3\}$ or $G$ satisfies the following conditions:
\begin{enumerate}
	\item $B_1,B_l\in \{C_3,K_2\}$,
	\item $B_2=B_{l-1}=C_4$,
	\item $B_i\in \{C_3,C_4\}$ for $3\leq i\leq l-2$ and if $B_i=C_3$ then $B_{i+1}=C_4$, and
	\item there are exactly two cut points in $C_4$ and they are adjacent.
\end{enumerate}
Our goal in this section is to compute the regularity of
binomial edge ideals of such classes of graphs. We compute the
regularity of a slightly more general class of graphs.
\begin{theorem}\label{CM-generalized-cactus-graph}
Let $G$ be a graph such that $\Ba(G)$ is a path of length $l-1$ for some
$l\geq 3$. Also let $B_1=K_{m_1},B_l=K_{m_l},B_2=B_{l-1}=C_4$ with
$m_1\geq 2,m_l\geq 3$ and $B_i\in \{C_4,K_{m_j}: m_j\geq 3\}$ for
$3\leq i\leq l-2$. Further assume that there are exactly two cut
points in each $C_4$ in $G$ and they are adjacent. Then $\reg(S/J_G)=2\cab(G)+\ca'(G)$, where $\ca'(G)$ is the number
of maximal cliques except the edges of $C_4$ in $G$.
\end{theorem}

\begin{proof}
We proceed by induction on $l$. If $l=3$, then $G=K_{m_1}\cup_
{v_{21}}C_4 \cup_{v_{24}} K_{m_3}$, and hence, the assertion
follows from Lemma \ref{tech-lemma2} considering $k=4$. Assume
that $l\geq 4$. Since $v_{23}$ is an internal vertex, by
Lemma \ref{ohtani-lemma}, we get the following short exact
sequence:
\begin{align}\label{ses-v23}
0\longrightarrow \frac{S}{J_{G}}\longrightarrow
\frac{S}{J_{{G_{v_{23}}}}}\oplus \frac{S}{Q_{v_{23}}}
\longrightarrow \frac{S}{J_{G_{v_{23}}}+Q_{v_{23}}}
\longrightarrow 0.
\end{align}
Set $r=2\cab(G)+\ca'(G)$. First, we show that $\reg(S/J_{G_{v_{23}}})=r$.
Note that $E(G_{v_{23}})=E(G)\cup \{ \{v_{22},v_{24}\}\}$, and so the
second block in $G_{v_{23}}$ is a diamond graph $D$. Except for the
second block, all the blocks of $G$ and $G_{v_{23}}$ are the
same.  Here, $G_{v_{23}}$ is a decomposable graph into
indecomposable graphs $K_{m_1}$ and $H$ along $v_{21}$, where
$H=D\cup_{w_2} B_3\cup_{w_3} \dots \cup_{w_{l-1}} B_l$. Now to
find the regularity of $J_{G_{v_{23}}}$, it is enough to find the
regularity of $J_H$.
\vskip 1mm
\noindent
\textbf{Claim:} $\reg(S_H/J_H)=r-1$.\\
Since $v_{22}$ is an internal vertex of $H$, by Lemma
\ref{ohtani-lemma}, we get the following exact sequence:
\begin{align}\label{ses-v22}
0\longrightarrow \frac{S_H}{J_{H}}\longrightarrow
\frac{S_H}{J_{H_{v_{22}}}}\oplus
\frac{S_H}{Q_{v_{22}}}\longrightarrow
\frac{S_H}{J_{H_{v_{22}}}+Q_{v_{22}}} \longrightarrow 0.
\end{align}
Note that $H_{v_{22}}=K_4\cup_{w_2} B_3\cup_{w_3} \dots \cup_{w_{l-1}}
B_l$. Since $\Ba(H_{v_{22}})$ is of length $l-1$ and $H_{v_{22}}$
satisfies induction hypotheses,  we have
$\reg(S_H/J_{H_{v_{22}}})=2\cab(H_{v_{22}})+\ca'(H_{v_{22}})=r-2$. It follows from \cite[Proposition
2.3]{Rinaldo-Cactus} that if $T\in
\mathscr{C}(H)$ and $v_{22} \in T$, then $v_{24} \in T$. So,
$Q_{v_{22}}=(x_{v_{22}},y_{v_{22}},x_{v_{24}},y_{v_{24}})+J_{H\setminus
\{v_{22},v_{24}\}}$ and
$J_{H_{v_{22}}}+Q_{v_{22}}=(x_{v_{22}},y_{v_{22}},x_{v_{24}},y_{v_{24}})+
J_{H\setminus
\{v_{22},v_{24}\}}+(x_{v_{21}}y_{v_{23}}-x_{v_{23}}y_{v_{21}})$. It can be seen that
$H\setminus \{v_{22},v_{24}\}=B_3'\cup_{w_3} \dots \cup_{w_{l-1}}
B_l\cup \{v_{21},v_{23}
\}$, where $B_3'$ is obtained by deleting the vertex $w_2=v_{24}$
from $B_3$. 
\vskip 1mm
\noindent
\textbf{Case 1.} If $B_3=K_{m_3}$, then $H\setminus
\{v_{22},v_{24}\}=K_{m_3-1}\cup_{w_3} \dots \cup_{w_{l-1}} B_l
\cup \{v_{21}, v_{23}\}$. Hence, by induction and \eqref{Bettiproduct},
$\reg(S_H/Q_{v_{22}})=r-3$ and
$\reg(S_H/(J_{H_{v_{22}}}+Q_{v_{22}}))=r-2$. 
\vskip 1mm
\noindent
\textbf{Case 2.} If $B_3=C_4$, then $H\setminus
\{v_{22},v_{24}\}=P_3\cup_{w_3} \dots \cup_{w_{l-1}} B_l \cup
\{v_{21}, v_{23}\}$.
Therefore, by \cite[Theorem 3.1]{JNR1}, induction and
\eqref{Bettiproduct}, $\reg(S_H/Q_{v_{22}})=r-3$ and
$\reg(S_H/(J_{H_{v_{22}}}+Q_{v_{22}}))=r-2$. Hence it follows
from Lemma \ref{regularity-lemma} and the short exact sequence
\eqref{ses-v22} that $\reg(S_H/J_H) \leq r-1$.

Now we prove that $r-1 \leq \reg(S_H/J_H)$.
Set $|V(H)|=n$. By \cite[Theorem 2.1]{Rinaldo-Cactus},
$S_H/J_{H_{v_{22}}}$, $S_H/J_{Q_{v_{22}}}$ and
$S_H/(J_{H_{v_{22}}}+Q_{v_{22}})$ are Cohen-Macaulay. Therefore,
$\beta_{n-1,n-1+r-2}(S_H/J_{H_{v_{22}}})$,
$\beta_{n-1,n-1+r-3}(S_H/J_{Q_{v_{22}}})$ and
$\beta_{n,n+r-2}(S_H/(J_{H_{v_{22}}}+Q_{v_{22}}))$ are the unique
extremal Betti numbers. We consider the long exact sequence of Tor
corresponding to the short exact sequence \eqref{ses-v22} in
homological degree $n$ and in graded degree $n+r-2$:
\begin{align*}
0 \longrightarrow \Tor_{n}^{S_H}\left(\frac{S_H}{J_{H_{v_{22}}}+Q_{v_{22}}},\K\right)_{n+r-2} \longrightarrow \Tor_{n-1}^{S_H}\left( \frac{S_H}{J_H},\K\right)_{n+r-2}\longrightarrow \cdots
\end{align*}
which implies that $\beta_{n-1,n+r-2}(S_H/J_H)\neq 0$ and hence, $r-1
\leq \reg(S_H/J_H)$. Therefore, $\reg(S_H/J_H)=r-1$.
By \cite[Theorem 3.1]{JNR1} and the claim, we get that $\reg(S/J_{G_{v_{23}}})=r$. 

Now we show that $\reg(S/Q_{v_{23}})\leq r-2$ and
$\reg(S/(Q_{v_{23}}+J_{G_{v_{23}}}))\leq r-1$. It follows from
\cite[Proposition 2.3]{Rinaldo-Cactus} that if $T\in \mathscr{C}(G)$
and $v_{23} \in T$, then $v_{21} \in T$ and $v_{24} \notin T$. Thus,
$Q_{v_{23}}=(x_{v_{21}},y_{v_{21}},x_{v_{23}},y_{v_{23}})+J_{K_{m_1-1}}+J_{H'}$,
where $H'=(B_3\cup_{w_3}B_4\cup_{w_4} \dots \cup_{w_{l-1}}
B_l)_{v_{24}}=(B_3)_{v_{24}}\cup_{w_3}B_4\cup_{w_4} \dots
\cup_{w_{l-1}} B_l$.
\vskip 1mm
\noindent
\textbf{Case 1.} If $B_3=K_{m_3}$, then $(B_3)_{v_{24}}=B_3$.
Therefore, $H'$ satisfies the induction hypotheses and hence,
$\reg(S_{H'}/J_{H'})=r-3$.
\vskip 1mm
\noindent
\textbf{Case 2.} If $B_3=C_4$, then $(B_3)_{v_{24}}=D$. So,
$H'=D\cup_{w_3}B_4\cup_{w_4} \dots \cup_{w_{l-1}} B_l$. It follows
from the Claim that $\reg(S_{H'}/J_{H'})=r-3$. 

\vskip 1mm
Hence, in either case, $\reg(S/Q_{v_{23}})\leq r-2$. Also
$J_{G_{v_{23}}}+Q_{v_{23}}=(x_{v_{21}},y_{v_{21}},x_{v_{23}},y_{v_{23}})+J_{K_{m_1-1}}+J_{H''}$,
where $E(H'')=E(H')\cup \{\{v_{22},v_{24}\}\}$. Therefore, by
\cite[Theorem 3.1]{JNR1},
$\reg(S/(J_{G_{v_{23}}}+Q_{v_{23}}))=\reg(S/Q_{v_{23}})+1\leq r-1$.
Hence, it follows from Lemma \ref{regularity-lemma} and the short
exact sequence \eqref{ses-v23} that $\reg(S/J_G)=r$.
		\end{proof}
As an immediate consequence, we obtain a class of cactus graphs for
which the upper bound obtained in Corollary \ref{cactus-graph}
is attained. 
\begin{corollary}\label{CM-cactus-graph}
If $G$ is a Cohen-Macaulay indecomposable cactus graph such that $B_1=C_3$ or
$B_l=C_3$, then $\reg(S/J_G)=2\cab(G)+\ca'(G)$. 
\end{corollary}
Below, we give examples of two Cohen-Macaulay cactus graphs. The graph $G_1$ is an example of
a graph for which $\reg(S_{G_1}/J_{G_1}) < 2\cab(G_1) + \ca'(G_1)$. The graph $G_2$
shows that the assumption $B_1 = C_3$ or $B_l = C_3$ is not a
necessary one for the equality. If one inputs the binomial edge
ideals corresponding to these two graphs into any of the computational
commutative algebra packages (we used Macaulay2, \cite{M2}) and
compute the regularity, then we get $\reg(S_{G_1}/J_{G_1}) = 6 =
\reg(S_{G_2}/J_{G_2})$. It may be noted that $2\cab(G_1) + \ca'(G_1) = 7$ and
$2\cab(G_2) + \ca'(G_2) = 6$.
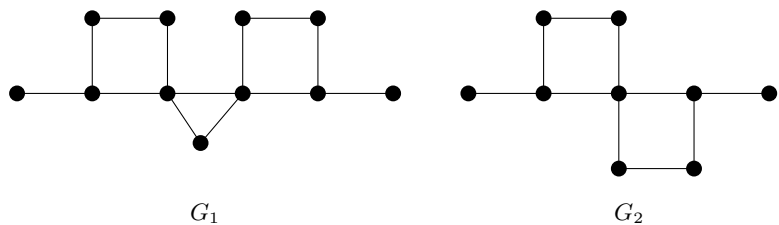
\usetikzlibrary{arrows}
\begin{figure}[H]
	\begin{tikzpicture}[scale=2]
\draw (0.5,0.5)-- (1,0.5);
\draw (1,0.5)-- (1,1);
\draw (1,1)-- (1.5,1);
\draw (1.5,1)-- (1.5,0.5);
\draw (1.5,0.5)-- (1,0.5);
\draw (1.5,0.5)-- (1.72,0.17);
\draw (1.72,0.17)-- (2,0.5);
\draw (2,0.5)-- (1.5,0.5);
\draw (2,0.5)-- (2,1);
\draw (2,1)-- (2.5,1);
\draw (2.5,1)-- (2.5,0.5);
\draw (2.5,0.5)-- (2,0.5);
\draw (2.5,0.5)-- (3,0.5);
\draw (3.5,0.5)-- (4,0.5);
\draw (4,0.5)-- (4,1);
\draw (4,1)-- (4.5,1);
\draw (4.5,1)-- (4.5,0.5);
\draw (4.5,0.5)-- (4,0.5);
\draw (4.5,0.5)-- (4.5,0);
\draw (4.5,0)-- (5,0);
\draw (5,0)-- (5,0.5);
\draw (5,0.5)-- (4.5,0.5);
\draw (5,0.5)-- (5.5,0.5);
\begin{scriptsize}
\fill (0.5,0.5) circle (1.5pt);
\fill (1,0.5) circle (1.5pt);
\fill (1,1) circle (1.5pt);
\fill (1.5,1) circle (1.5pt);
	\draw(1.75,-0.3) node{$G_1$};
	\draw(4.57,-0.3) node{$G_2$};
\fill (1.5,0.5) circle (1.5pt);
\draw(1.57,0.62);
\fill (1.72,0.17) circle (1.5pt);
\draw(1.79,0.28);
\fill (2,0.5) circle (1.5pt);
\draw(2.07,0.62);
\fill (2,1) circle (1.5pt);
\draw(2.07,1.12);
\fill (2.5,1) circle (1.5pt);
\draw(2.55,1.12);
\fill (2.5,0.5) circle (1.5pt);
\draw(2.55,0.62);
\fill (3,0.5) circle (1.5pt);
\draw(3.07,0.62);
\fill (3.5,0.5) circle (1.5pt);
\draw(3.56,0.62);
\fill (4,0.5) circle (1.5pt);
\draw(4.08,0.62);
\fill (4,1) circle (1.5pt);
\draw(4.07,1.12);
\fill (4.5,1) circle (1.5pt);
\draw(4.57,1.12);
\fill (4.5,0.5) circle (1.5pt);
\draw(4.57,0.62);
\fill (4.5,0) circle (1.5pt);
\draw(4.57,0.11);
\fill (5,0) circle (1.5pt);
\draw(5.07,0.11);
\fill (5,0.5) circle (1.5pt);
\draw(5.07,0.62);
\fill (5.5,0.5) circle (1.5pt);
\draw(5.56,0.62);
\end{scriptsize}
\end{tikzpicture}
\end{figure}

\bibliographystyle{plain}
\bibliography{Reference}

\begin{thebibliography}{10}

\bibitem{EZ}
Viviana Ene and Andrei Zarojanu.
\newblock On the regularity of binomial edge ideals.
\newblock {\em Math. Nachr.}, 288(1):19--24, 2015.

\bibitem{M2}
Daniel~R. Grayson and Michael~E. Stillman.
\newblock Macaulay2, a software system for research in algebraic geometry.
\newblock Available at \url{http://www.math.uiuc.edu/Macaulay2/}.

\bibitem{haray}
F.~Harary.
\newblock {\em Graph Theory}.
\newblock Addison-Wesley Publishing Company. Boston, 1969.

\bibitem{HH1}
J\"urgen Herzog, Takayuki Hibi, Freyja Hreinsd\'ottir, Thomas Kahle, and
  Johannes Rauh.
\newblock Binomial edge ideals and conditional independence statements.
\newblock {\em Adv. in Appl. Math.}, 45(3):317--333, 2010.

\bibitem{JA1}
A.~V. Jayanthan and Arvind Kumar.
\newblock Regularity of binomial edge ideals of {C}ohen-{M}acaulay bipartite
  graphs.
\newblock {\em Comm. Algebra}, 47(11):4797--4805, 2019.

\bibitem{JAR2}
A.~V. Jayanthan, Arvind Kumar, and Rajib Sarkar.
\newblock Regularity of powers of quadratic sequences with applications to
  binomial ideals.
\newblock {\em J. Algebra}, 564:98--118, 2020.

\bibitem{JNR1}
A.~V. Jayanthan, N.~Narayanan, and B.~V. Raghavendra~Rao.
\newblock Regularity of binomial edge ideals of certain block graphs.
\newblock {\em Proc. Indian Acad. Sci. Math. Sci.}, 129(3):Art. 36, 10, 2019.

\bibitem{Kahle19}
Thomas {Kahle} and Jonas {Kr{\"u}semann}.
\newblock {Binomial edge ideals of cographs}.
\newblock {\em arXiv e-prints}, page arXiv:1906.05510, June 2019.

\bibitem{KM-arxiv13}
Dariush {Kiani} and Sara {Saeedi Madani}.
\newblock {The regularity of binomial edge ideals of graphs}.
\newblock {\em arXiv e-prints}, page arXiv:1310.6126, Oct 2013.

\bibitem{Arv-Jaco}
Arvind {Kumar}.
\newblock {Binomial edge ideals and bounds for their regularity}.
\newblock {\em J. Algebraic Combin.}, (To Appear).

\bibitem{MM}
Kazunori Matsuda and Satoshi Murai.
\newblock Regularity bounds for binomial edge ideals.
\newblock {\em J. Commut. Algebra}, 5(1):141--149, 2013.

\bibitem{oh}
Masahiro Ohtani.
\newblock Graphs and ideals generated by some 2-minors.
\newblock {\em Comm. Algebra}, 39(3):905--917, 2011.

\bibitem{Rinaldo-Cactus}
Giancarlo Rinaldo.
\newblock Cohen-{M}acaulay binomial edge ideals of cactus graphs.
\newblock {\em J. Algebra Appl.}, 18(4):1950072, 18, 2019.

\bibitem{KM18}
M.~{Rouzbahani Malayeri}, S.~{Saeedi Madani}, and D.~{Kiani}.
\newblock {Regularity of binomial edge ideals of chordal graphs}.
\newblock {\em Collect. Math. (To Appear)}, 2020.

\bibitem{KM1}
Sara Saeedi~Madani and Dariush Kiani.
\newblock Binomial edge ideals of graphs.
\newblock {\em Electron. J. Combin.}, 19(2):Paper 44, 6, 2012.

\bibitem{KM2}
Sara Saeedi~Madani and Dariush Kiani.
\newblock On the binomial edge ideal of a pair of graphs.
\newblock {\em Electron. J. Combin.}, 20(1):Paper 48, 13, 2013.

\bibitem{Rajib}
Rajib {Sarkar}.
\newblock {Binomial Edge Ideals of Unicyclic Graphs}.
\newblock {\em arXiv e-prints}, page arXiv:1911.12677, November 2019.

\end{thebibliography}

\end{document}